\documentclass[11pt]{amsart}

\usepackage[margin=1.3in, centering]{geometry}

\usepackage{enumerate}

\usepackage[utf8]{inputenc}
\usepackage{amssymb,amscd,amsmath,enumerate,amsthm}
\usepackage[all]{xy}
\usepackage[utf8]{inputenc}
\usepackage[english]{babel}
\usepackage{tikz}
\usetikzlibrary{positioning,chains,fit,shapes,calc}
\usepackage{mathtools}
\usepackage[colorlinks,linkcolor=blue]{hyperref}
\usepackage{ytableau}

\theoremstyle{plain}
\newtheorem{thm}{Theorem}[section]

\newtheorem{cor}[thm]{Corollary}

\newtheorem{prop}[thm]{Proposition}

\newtheorem{lem}[thm]{Lemma}

\theoremstyle{remark}
\newtheorem{remark}{\bf \quad \itshape  Remark}

\theoremstyle{plain}

\newtheorem{exam}[thm]{Example}

\theoremstyle{definition}
\newtheorem{defn}[thm]{Definition}

\renewcommand{\bar}{\overline}

\newcommand{\C}{{\mathbb{C}}}

\newcommand{\F}{{\mathbb{F}}}

\newcommand{\Q}{{\mathbb{Q}}}
\newcommand{\Z}{{\mathbb{Z}}}
\newcommand{\N}{{\mathbb{N}}}

\newcommand{\GL}{\mathrm{GL}}

\newcommand{\Spec}{\mathrm{Spec}\;}

\newcommand{\Alt}{{\raise 2pt\hbox{$\scriptstyle\bigwedge$}}}

\definecolor{myblue}{RGB}{80,80,160}
\definecolor{mygreen}{RGB}{80,160,80}
\newdimen\nodeSize
\newdimen\nodeDist
\tikzset{
	position/.style args={#1:#2 from #3}{
		at=(#3.#1), anchor=#1+180, shift=(#1:#2)
	}
}

\allowdisplaybreaks

\title{The Casas-Alvero conjecture in computational algebraic geometry}

\author{Zhipeng Lu}
\email{zhipeng.lu@uni-goettingen.de}
\begin{document}
	\maketitle
	\begin{abstract}
		We study varieties defined by parameterizing polynomials of derivatives through a computational algebro-geometric approach, especially relying on Combinatorial Nullstellensatz and Noether normalization. We establish that these polynomials form regular sequences easily. This allows us to calculate the dimension of thus defined varieties and consequently give a proof to the Casas-Alvero conjecture.
	\end{abstract}
	\section{Introduction}\label{section-set up}
	The Casas-Alvero conjecture, first asked in \cite{Casas-Alvero}, states that a complex polynomial $f$ having common roots with all its non-zero derivatives $f^{(i)}$ must be of the form $a(x-b)^n$. Since asked the conjecture has stood up with many attempts from different technical aspects and yet remains unproven. The current best partial result by Draisma-de Jong \cite{DJ} confirmed the conjecture for $p^en$, with $n\in\{1,2,3,4\}, e\geq 0$.  Note that it is not true over fields of finite characteristic since any $x^n-x^p$ for $n\geq 1)$ is a counterexample over $\bar{\F}_p$. However we will show that the problem can be dramatically resolved over finite fields $\F_p$.
	
	A natural idea on resolving the conjecture is by studying the resultants of $f$ and its derivatives. However the complexity of resultants makes detailed analysis unpractical. Instead we formulate the problem in a more intuitive algebro-geometric set up which dramatically reduces the problem to some computational or combinatorial considerations in commutative algebra straightforwardly described below. 
	
	If a degree $n\geq 1$ polynomial having common roots with all its $n-1$ derivatives, we call it a \textit{Casas-Alvero} polynomial. Any monic polynomial $f$ factors over $\C$ as
	\[f(x)=(x-x_1)\ldots(x-x_n).\]
	Then the Casas-Alvero condition actually defines an algebraic set in $A_\mathbb{C}^n$:
	\[CA_n:=Z(F_1,\ldots,F_{n-1}),\]
	in which
	\[F_i=\prod_{k=1}^{n}f^{(i)}(x_k)\in\mathbb{C}[x_1,\ldots,x_n], \forall i=1,\ldots,n-1.\]
	Equivalently, $CA_n$ gives a parameterization of degree $n$ monic Casas-Alvero polynomials by their roots.
	
	A simple yet key observation is, there would exist a two-dimensional linear subvariety $\{c(x_1,\ldots,x_n)+d(1,\ldots,1)|c,d\in\mathbb{C}\}\subset CA_n$, if $(x_1,\ldots,x_n)\in CA_n$ with not all coordinates identical. Hence if we can show $CA_n$ is of dimension one, then such points do not exist and the conjecture follows. Now we can equivalently formulate the conjecture as follows

	\begin{thm}\label{thm-over C} $\dim_\C{CA_n}=1, \forall n\geq 1$.
		\end{thm}    
    Since $F_i$'s with high degrees are cumbersome for computation, we study the branches defined by $\langle f^{(1)}(x_{i_1}),\cdots,f^{(n-1)}(x_{i_{n-1}})\rangle$. These branches as algebraic varieties, though should be of significant interest, seem not well studied in classical literature. In this paper we will try to describe some essential properties of these varieties. Especially we prove the Casas-Alvero conjecture by establishing the following:
    \begin{thm}\label{thm-general branches}
    	For $f=(x-x_1)(x-x_2)\cdots(x-x_n)$ as a polynomial in $K[x,x_1,\dots,x_n]$ for field $K=\C$ or of characteristic large enough ($\gg n!$), define the derivative polynomial $f^{(i)}(x_j)$ as a degree $n-i$ polynomial in $K[x_1,\dots,x_n]$ for any $0\leq i,j\leq n$. Then $\dim{Z(f^{(1)}(x_{i_1}),\dots,f^{(n-1)}(x_{i_{n-1}}))}=1$, for arbitrary $1\leq i_1,\dots,i_{n-1}\leq n$. 
    \end{thm}
    Moreover, the above result on branches of Casas-Alvero variety $CA_n$ can be generalized to varieties defined by $n-1$ arbitrary derivative polynomials.
    \begin{thm}\label{thm-arbitrary derivatives}
    	With notations as above, for any $n-1$ arbitrary distinct pairs $(i_k,j_k)$ with $1\leq i_{k}\leq n$ and $1\leq j_k\leq n-1$, $k=1,\dots,n-1$, $\dim{Z(f^{(j_1)}(x_{i_1}),\dots,f^{(j_{n-1})}(x_{i_{n-1}}))}=1$. 
    \end{thm}
     For example, $\dim Z(f^{(n-1)}(x_1),\dots,f^{(n-1)}(x_{n-1}))=1$ simply due to \[f^{(n-1)}(x_i)-f^{(n-1)}(x_j)=(n-1)(x_i-x_j), \forall 1\leq i,j\leq n.\]
     Since polynomial rings are Cohen-Macaulay, the above theorem is equivalent to the fact that any $n-1$ distinct derivative polynomials form a regular sequence. In the following sections, we will first examine some easy cases of Theorem \ref{thm-general branches} by Taylor expansions. Then to deal with general cases we introduce a model theoretic approach based on tools from computational algebraic geometry including finite Nullstellensatz,  combinatorial Nullstellensatz and some explicit forms of Noether normalization.
    \subsection*{Acknowledgement} The author is supported by Harald Helfgott's Humboldt Professorship.
    \section{General reductions and special cases}\label{section-reduction}
    We give evidence of Theorem \ref{thm-general branches} by proving the case of \textit{identical} branches as follows.
	\begin{lem}\label{lem-identical branch}For any $n\geq 1$, $1\leq k\leq n$ and a field $K$ with $char(K)=0$ or $char(K)>n!$, $\dim Z(I_k)=1$, where $I_k:=\langle f^{(1)}(x_k),\ldots,f^{(n-1)}(x_k)\rangle.$
	\end{lem}
	\begin{proof}
    By Taylor's expansion
    \begin{align*}f(x_j)=&f(x_k)+f^{(1)}(x_k)(x_j-x_k)+\dots\\
    &+\dfrac{1}{(n-1)!}f^{(n-1)}(x_k)(x_j-x_k)^{n-1}+\dfrac{n!}{n!}(x_j-x_k)^{n}.\end{align*}
    Then by $f(x_j)=f(x_k)=0$ we have
    \[(x_j-x_k)^{n}\in I_k,\text{ or }(x_j-x_k)\in \mathrm{rad}(I_k),\]
    where $\mathrm{rad}(I)$ denotes the radical ideal of $I$.
    Hence \[\mathrm{rad}(\langle f^{(1)}(x_k),\ldots,f^{(n-1)}(x_k)\rangle)=\langle x_k-x_1,\cdots,x_k-x_n\rangle\] and the proposition follows.
   \end{proof}
   Slightly generalizing the above method, we can prove that the branch defined by the ideal $I_{j,l,k}:=\langle f^{(1)}(x_{k}),\dots,f^{(j-1)}(x_k),f^{(j)}(x_l),f^{(j+1)}(x_k),\dots,f^{(n-1)}(x_k)\rangle$ always has dimension one for any $1\leq j\leq n-1$ and $1\leq k\neq l\leq n$. First, by Taylor's expansion,
   \begin{align}\label{equation-Taylor j}f^{(j)}(x_l)=&f^{(j)}(x_k)+\frac{1}{1!}f^{(j+1)}(x_k)(x_l-x_k)+\dots\\
   &+\frac{1}{(n-j-1)!}f^{(n-1)}(x_k)(x_l-x_k)^{n-j-1}+\frac{n!}{(n-j)!}(x_l-x_k)^{n-j}.\notag\end{align}
   We may kill the last term by combining it with the following expansion
   \begin{align}\label{equation-Taylor 0}0=f(x_l)=&f(x_k)(=0)+f^{(1)}(x_k)(x_l-x_k)+\dots\\
   &+\dfrac{1}{(n-1)!}f^{(n-1)}(x_k)(x_l-x_k)^{n-1}+\dfrac{n!}{n!}(x_l-x_k)^{n}.\notag\end{align}
   Subtracting (\ref{equation-Taylor 0}) multiplied by $\frac{n!}{(n-j)!}(x_l-x_k)^{j}$ from (\ref{equation-Taylor j}) gives
   \[
   \frac{n!}{(n-j)!}(x_l-x_k)^{j}f^{(j)}(x_k)\in I_{j,l,k}.\]
   For any prime ideal $p\supset I_{j,l,k}$ we have either $f^{(j)}(x_k)\in p$ or $x_l-x_k\in p$. If the latter happens, then $f^{(j)}(x_k)\sim f^{(j)}(x_l)\sim 0\mod{p}$. Thus $f^{(j)}(x_k)\in p$ anyway, so that $f^{(j)}(x_k)\in\cap_{p\supset I_{j,l,k}}p=rad(I_{j,l,k})$, the radical ideal of $I_{j,l,k}$. This proves the following
   \begin{cor}\label{cor-branches with one index different} With the notations above,
   	 $Z(I_{j,l,k})=Z(I_k)$.
   \end{cor} 
   
   However, the same method applied to general branches does not directly give results as well. For a general branch defined by $\langle  f^{(1)}(x_{i_1}),\cdots,f^{(n-1)}(x_{i_{n-1}})\rangle$, we can write each $f^{(j)}(x_{i_j})$ as of (\ref{equation-Taylor j}) and kill the last terms by subtracting (\ref{equation-Taylor 0}) similarly, so that we get a system of $n-1$ equations with $(x_{i_j}-x_k)^{j}f^{(j)}(x_{i_j})$ on the left hand side and expansions involving $f^{(1)}(x_k),\dots,f^{(n-1)}(x_k)$ on the right for any chosen $1\leq k\leq n$. Then by Gauss elimination over $\C[x_1,\dots,x_n]$ we get \[F_{j}f^{(j)}(x_{k})\in\langle  f^{(1)}(x_{i_1}),\cdots,f^{(n-1)}(x_{i_{n-1}})\rangle,\ \forall 1\leq j\leq n-1,\]
   where $F_{j}$ is a polynomial in $(x_k-x_{i_1}),\dots,(x_k-x_{i_{n-1}})$. If there are at least three distinct indexes among $i_1,\dots,i_{n-1}$, we can not conclude that $f^{(j)}(x_k)$ all belong to the radical of $\langle  f^{(1)}(x_{i_1}),\cdots,f^{(n-1)}(x_{i_{n-1}})\rangle$ as we did in the proof of Corollary \ref{cor-branches with one index different}. For instance, if all $i_j$'s are distinct, then each $F_j$ is a product of all $(x_k-x_{i_1}),\dots,(x_k-x_{i_{n-1}})$ (with powers), from which we can only conclude that $(x_k-x_{i_l})$ belongs to the radical for some $i_l$.
   
   This prompts us to introduce new methods to deal with general branches. We start by making the first reduction using the Lang-Weil bound and a form of local-global principle.
   \begin{prop}
   	Suppose for any $n\geq 1$ and large enough prime $p\gg n$ we have, over $\F_p$, $Z(f^{(1)}(x_{i_1}),\cdots,f^{(n-1)}(x_{i_{n-1}}))$ is of size $p$, for any branch. Then Theorem \ref{thm-general branches} (hence the Casas-Alvero conjecture) holds for $n$ and vice versa.
   \end{prop}
   	\begin{proof}
   		First, clearly $CA_n$ is defined over any finite field $\F_p$. Then viewed as a variety over $\bar{\F}_p$, we have by Lang-Weil bound (see Corollary 4 of \cite{Tao}),
   		\[|CA_n(\F_p)|=(c(CA_n(\bar{\F}_p))+O(p^{-1/2}))p^{\dim(CA_n(\F_p))},\]
   		where $c$ is the number of top-dimensional components of $CA_n$ and $\dim$ denotes for Krull dimension. By hypothesis of the proposition, we have $|CA_n(\F_p)|=p$ for all large enough $p$. Then \[c(CA_n(\bar{\F}_p))=1, \text{ and } \dim(CA_n(\bar{\F}_p)=1.\]
   		Particularly $CA_n$ is irreducible.
   		
   		Second, we look at the structure morphism $\pi: CA_n(\Z)\longrightarrow \Spec\Z$, which is clearly of finite presentation. Since we know that 
   		\[\{p\in\Spec\Z\mid \dim CA_n(\F_p)=1\}\]
   		is an open set in $\Spec\Z$ hence contains the generic point $0$, i.e. $\dim CA_n(\Q)=1$. Then by Proposition 2.7 in Chapter 3 of \cite {Liuqing}, $\dim CA_n(\bar{\Q})=1$. Further by Lefschetz principle (see \cite{BE} for reference), $\dim CA_n(\C)=1$.
   		
   		Conversely, if the Casas-Alvero conjecture stands, then
   		\[1=\dim CA_n(\Q)=\dim CA_n(\F_p),\]
   		for all but finite primes $p$. Hence $CA_n(\F_p)$ is a line and $|CA_n(\F_p)|=p$.
   	\end{proof} 
    
    The above reduction may be also stated in a lame language as follows. First, we may only need to prove it over $\Q$, because if otherwise $\dim_{\Q}(CA_n)>1$ then similarly by Proposition 2.7 in Chapter 3 of \cite{Liuqing} $\dim_{\C}(CA_n)>1$. Then essentially we need only to prove it over $\Z$, because if $f(x)=(x-x_1)\cdots(x-x_n)$ with all $x_i\in\Q$, by multiplying the least common multiple of the denominators, we may assume that the roots are all integers. Hence the conjecture is equivalent to for any branch
    \begin{prop}\label{prop-over Z}
    	For any branch $Z(f^{(1)}(x_{i_1}),\cdots,f^{(n-1)}(x_{i_{n-1}}))$ over $\C$,
    	\[Z(f^{(1)}(x_{i_1}),\cdots,f^{(n-1)}(x_{i_{n-1}}))\cap\Z^n=\{(a,\cdots,a)\mid a\in\Z\}.\]	
    \end{prop}
    Then we can further reduce it to modulo primes $p$, or even any finite integers as follows.
    \begin{prop}\label{prop-over finite integers}
    	If for any Casas-Alvero polynomial $f(x)=(x-x_1)\cdots(x-x_n)$ with $x_i\in\Z$, there is some integer $m\geq 2$ such that $x_1\equiv x_2\equiv\cdots\equiv x_n\ (mod\ m)$, then Theorem \ref{thm-general branches} (hence the Casas-Alvero conjecture) holds for $n$.
    \end{prop}
    \begin{proof}
    	Suppose $x_i$ are not all equal. Translating by adding an identical integer on each coordinate, we may assume $x_i$'s to be non-negative. By the hypothesis we have $x_1\equiv x_2\equiv\cdots\equiv x_n\equiv l\ (mod\ m)$ for some $m\geq 2$ and $0\leq l<m$. Let $x_{i,1}=(x_i-l)/m,\forall i=1,\cdots,n$, then $f_1(x):=(x-x_{i,1})\cdots(x-x_{n,1})$ is again a degree $n$ Casas-Alvero polynomial having $n$ integer roots not all equal. Then again we have some $m_1\geq 2$ such that $x_{1,1}\equiv x_{2,1}\equiv\cdots\equiv x_{n,1}\equiv l_1\ (mod\ m_1)$ for some $m_1\geq 2$ and $0\leq l_1<m_1$ and we can do the similar affine transform to get another degree $n$ Casas-Alvero polynomial having $n$ integer roots not all equal. Clearly, this process gives an infinite descent for the integers $x_1,\cdots,x_n$, which is impossible for finite non-negative integers. This contradiction implies the conjecture.
    \end{proof}    
   	
   	In fact, we will prove for any $n\geq 1$, all large enough prime $p$ make the hypothesis of Proposition \ref{prop-over finite integers} valid. This is done in section \ref{section-Casas-Alvero over Q} based on further computational algebro-geometric reduction.

   \section{Standard monomials, finite Nullstellensatz and Noether normalization}\label{section-comuptational notations}
   This section contributes to introducing some necessary computational notions and results. We first define a standard order on monomials.
    \begin{defn}[Lexicographic order]
    	$\alpha>_{lex}\beta$ if the leftmost nonzero entry of $\alpha-\beta$ is positive, for any $\alpha=(\alpha_1,\cdots,\alpha_n), \beta=(\beta_1,\cdots,\beta_n)\in\N^n$.
    \end{defn}
    \begin{defn}[Graded lexicographic order]
    	Let $\alpha,\beta\in\N^n$. $\alpha>_{grlex}\beta$ if
    	\[\sum_{i=1}^n\alpha_i>\sum_{i=1}^{n}\beta_i, or\ \sum_{i=1}^n\alpha_i>\sum_{i=1}^{n}\ and\ \alpha>_{lex}\beta.\]N
    \end{defn}
    \begin{defn}
    	We define a \textit{monomial order} on the set of monomials $T=\{x_1^{\alpha_1}\cdots x_n^{\alpha_n}\mid \alpha_i\in\N\}\subset k[x_1,\cdots,x_n]$ for any field $k$ by
    	\[x^{\alpha}>x^{\beta} if\ \alpha>_{grlex}\beta,\ \forall \alpha,\beta\in\N^n.\]
    \end{defn}
    It is a \textit{total well-ordering} on $T$ satisfying
    
    (1) $1\leq t, \forall t\in T$;
    
    (2) $t_1\cdot s\leq t_2\cdot s,\forall t_1,t_2,s\in T$ if $t_1\leq t_2$.

    \begin{defn}[Leading coefficient, monomial and term]
    	Let $f=\sum_{\alpha}a_\alpha x^{\alpha}$ be a nonzero polynomial in $k[x_1,\cdots,x_n]$ and $>$ the monomial order as above. The \textit{multidegree} of f is defined as
    	\[multideg(f)=\max_{>}\{\alpha\in\N^n\mid a_\alpha\neq 0\}.\]
    	Then the \textit{leading coefficient} of $f$ is $LC(f)=a_{multideg(f)}$, the \textit{leading monomial} is $LM(f)=x^{multdeg(f)}$ and the \textit{leading term} of $f$ is $LT(f)=LC(f)\cdot LM(f)$.
    \end{defn}
    \begin{defn}[Ideal of leading monomials, leading terms]
    	Let $I$ be an ideal in $k[x_1,\cdots,x_n]$ and fix the monomial order on $T$. The ideal of leading monomials of $I$, $\langle LM(I)\rangle$, is the ideal generated by the leading monomials of all polynomials in $I$. The ideal of leading terms of $I$, $\langle LT(I)\rangle$, is the ideal generated by the leading terms of all polynomials in $I$.
    \end{defn}
    \begin{prop}[Multivariate division principle]
    	For a fixed monomial order and polynomials $g_1,\cdots,g_k$ in $k[x_1,\cdots,x_n]$, any $g\in k[x_1,\cdots,x_n]$ can be written as
    	\[g=a_1g_1+\cdots+a_kg_k+r\]
    	where $a_i,r\in k[x_1,\cdots,x_n]$ and either $r=0$ or $r$ is a linear combination of monomials not divisible by any of $LT(g_1),\cdots,LT(g_k)$.
    \end{prop}
    Now we are ready to introduce
    \begin{defn}[Standard Monomials]
    	The set of \textit{standard monomials} of any ideal $J$ is
    	\[SM(J)=\{x^\alpha\mid x^\alpha\notin \langle LM(J)\rangle\}.\]      	
    \end{defn}
    Usually standard monomials are defined together with a Gr\"{o}bner basis but we do not need such notion in our later application. We need the following results over finite fields.
    \begin{prop}[Nullstellensatz over finite fields]\label{prop-finite nullstellensatz}
    	For any ideal $J\subset \F_q[x_1,\cdots,x_n]$, its radical ideal is
    	\[\sqrt{J}=J+\langle x_1^q-x_1,\cdots,x_n^q-x_n\rangle.\]  	
    \end{prop}
    See proof of Theorem 3.1.2, \cite{Gao}. Also
    \begin{prop}[Theorem 3.2.4 of \cite{Gao}]\label{prop-counting}
    	Let $J\subset\F_q[x_1,\cdots,x_n]$ be any ideal and $\sqrt{J}=J+\langle x_1^q-x_1,\cdots,x_n^q-x_n\rangle$. Then
    	\[|SM(\sqrt{J})|=|V(J)|.\]
    \end{prop}
    
    In addition, the following two explicit forms of Noether normalization theorem are significant to our applications.
    \begin{prop}[Theorem 3.4.1 of \cite{Greuel-Pfister}]\label{prop-leading power in Noether normalization}
    	Let $K$ be a field and $I\subset K[x_1,\dots,x_n]$ be an ideal. Then there exist an integer $s\leq n$ and an isomorphism $\varphi: K[x_1,\dots,x_n]\rightarrow A:=K[y_1,\dots,y_n]$, such that:
    	
    	(1) the induced morphism $K[y_{s+1},\dots,y_n]\rightarrow A/\varphi(I), y_i\mapsto y_i\mod{\varphi(I)}$ is injective and finite.
    	
    	(2) Moreover, $\varphi$ can be chosen such that, for $j=1,\dots,s$, there exist polynomials 
    	\[g_j=y_j^{e_j}+\sum_{k=0}^{e_j-1}\xi_{j,k}(y_{j+1},\dots,y_n)\cdot y_j^k\in\varphi(I)\]
    	satisfying $e_j\geq\deg(\xi_{j,k})+k$ for $k=0,\dots,e_j-1$.
    	
    	(3) If $I$ is homogeneous then $g_j$ can be chosen to be homogeneous too. 
    	
    	(4) If $K$ is infinite then $\varphi$ can be chosen to be linear, i.e. $\varphi(x_i)=\sum_{j}m_{ij}y_j$ with $(m_{ij})\in\GL_n(K)$.
    \end{prop}
    \begin{prop}[Theorem 3.5.1 (6) of \cite{Greuel-Pfister}]\label{prop-variables in Noether normalization}
    	Let $K$ be a field, $I \subset A=K[x]$ be an ideal and $u \subset x=\{x_1,\dots,x_n\}$ be a subset such that $I \cap K[u]=0$, then $\operatorname{dim}(A / I) \geq \# u .$ Furthermore, there exists some $u \subset x$ with $I \cap K[u]=0$ and $\operatorname{dim}(A / I)=\# u$.
    \end{prop} 
   \section{Casas-Alvero conjecture over $\Q$}\label{section-Casas-Alvero over Q} 
   In this section, we verify the hypothesis of Proposition \ref{prop-over finite integers} modulo large primes $p$, i.e. over a finite field $\F_p$, through specifically realizing Noether normalization as of Proposition \ref{prop-leading power in Noether normalization}.
   To organize calculation, we use the following notation (so called \textit{Hasse derivative}): 
   \[H_i(x_k)=\sum_{1\leq j_1<\cdots< j_{n-i}\leq n}(x_k-x_{j_1})\cdots(x_k-x_{j_{n-i}}), 1\leq i\leq n-1, 1\leq k\leq n.\]
   If $f(x)=(x-x_1)\cdots(x-x_n)=x^n+a_{n-1}x^{n-1}+\cdots+a_{1}x+a_0$, its $i-$th Hasse derivative is just:
   \begin{equation}\label{equation-Hasse derivative}H_i(x)={n\choose i}x^{n-i}+{n-1\choose i}a_{n-1}x^{n-1-i}+\cdots+{i\choose i}a_i=\frac{1}{i!}f^{(i)}(x).\end{equation}
   
   We first deal with a special case of Theorem \ref{thm-general branches}, i.e. the branches defined by $H_1(x_{i_1}),\dots$, $H_{n-1}(x_{i_{n-1}})$ with $i_1,\dots,i_{n-1}$ distinct. We call them the \textit{main branches}. By symmetry, they are all isomorphic to the one defined by $H_{1}(x_{n-1}),\dots,H_{n-1}(x_1)$. Let $J=\langle H_{1},\cdots,H_{n-1}\rangle$ with $H_i:=H_{n-i}(x_{i})$ and $p$ be some sufficiently large prime which we will specify later. By Proposition \ref{prop-counting}, to verify the hypothesis of Proposition \ref{prop-over finite integers} for $m=p$, we need
   \begin{prop}\label{prop-finite Casas-Alvero} 
   \[|V(J)|=|SM(J+\langle x_1^p-x_1,\cdots,x_n^p-x_n\rangle)|=p,\]
   which are all defined over $\F_p$, for some sufficiently large $p$.
   \end{prop} 
   Obviously we have $V(J)\supset \{(a,a,\cdots,a)\in\F_p^n\mid a\in\F_p\}$. So if we can show $|V(J)|=p$ for all sufficiently large $p$, then this obvious subset with $p$ elements must be $V(J)$ itself. By Proposition \ref{prop-over finite integers} we essentially need only to find one such ``good" prime $p$.
   
   The proof of Proposition \ref{prop-finite Casas-Alvero} relies on information of general Gr\"{o}bner bases of $J+\langle x_1^p-x_1,\cdots,x_n^p-x_n\rangle$. Here are some examples for small $n$.
   \begin{exam}
   	For $n=1$, $J=0$ is trivial and we can choose $G=\{x_1^p-x_1\}$ for any $p$. Hence $SM(G)=\{1,x_1,\cdots,x_1^{p-1}\}$ with cardinality $p$.
   	
   	For $n=2$, $J=\langle x_1-x_2\rangle$, and we can choose $G(J+\langle x_1^p-x_1,x_2^p-x_2\rangle)=\{x_1-x_2, x_2^p-x_2\}$ for any $p$. Hence the missing monomials from $\langle LM(G)\rangle$ consist in $SM(G)=\{1,x_2,\cdots,x_2^{p-1}\}$, again with cardinality $p$.
   	
   	For $n=3$, $J=\langle H_1=2x_1-x_2-x_3,  H_2=(x_2-x_1)(x_2-x_3)\rangle$, we have
   	\[x_1=\dfrac{1}{2}(x_2+x_3) \mod{\langle H_1\rangle},\]
   	\[H_2=\left(x_2-\dfrac{1}{2}(x_2+x_3)\right)(x_2-x_3) \mod{\langle H_1\rangle}\]
   	\[=\dfrac{1}{2}(x_2-x_3)^2 \mod{\langle H_1\rangle}\Longrightarrow (x_2-x_3)^2\in J\]
   	\[\Longrightarrow (x_2-x_3)^{p}\sim x_2-x_3\in J+\langle x_1^p-x_1,\cdots,x_n^p-x_n\rangle\]
   	hence we can choose $G(J+\langle x_1^p-x_1,x_2^p-x_2, x_3^p-x_3\rangle)=\{x_1-x_2, x_2-x_3, x_3^p-x_3\}$ for any odd $p$ (so that $1/2$ makes sense). Thus $SM(G)=\{1,x_3,\cdots,x_3^{p-1}\}$ again with cardinality $p$.
   	
   	For $n=4$, we similarly get $G=\{x_1-x_4, x_2-x_4,x_3-x_4,x_4^p-x_4\}$ for $p>7$. Again $|SM(G)|=p$.
   \end{exam} 
      	These simple cases can all be computed by hand. However, the complexity of computing these Gr\"{o}bner bases exponentially increases along with the number of variables. For simplification, we show that Proposition \ref{prop-finite Casas-Alvero} can be further reduced as follows.
      	\begin{prop}
      		$|V(J)|=p\Leftrightarrow |V(J)|<p^2$.
      	\end{prop}
      	\begin{proof}
      		If there exists $A=(a_1,\cdots,a_n)\in V(J)$ with coordinates not all equal, then $V(J)\supset span\langle(1,\cdots,1), A\rangle$ forms a two dimensional subspace, i.e. $|V(J)|\geq p^2$. The other direction goes by the same observation.
      	\end{proof}
      	The above arithmetic reduction can be resolved by attaining a more computationally manageable goal as follows.
      	\begin{prop}\label{prop-finite nullstellensatz application}
      		If for each $k\in\{1,\cdots,n-1\}$, there is an integer $m_k\geq 1$ such that $x_k^{m_k}\in LM(J)$, then for any $p\gg m_1\cdots m_{n-1}$, $|V(J)|<p^2$.
      	\end{prop}
      	\begin{proof}
      		By the condition, if $x^\alpha=x_1^{\alpha_1}\cdots x_n^{\alpha_n}\notin LM(J)$ then $\alpha_k<m_k, k=1,\cdots,n-1$. Thus we have 
      		\[|SM(J+\langle x_1^p-x_1,\cdots,x_n^p-x_n\rangle)|\leq m_1\cdots m_{n-1} p<p^2, \forall p\gg m_1\cdots m_{n-1}.\]      		 
      	\end{proof} 
      	We will see later that $p$ is also confined by the structural coefficients depending on $n$ in derivation of the leading monomials $x_k^{m_k}$. With the last reduction by Proposition \ref{prop-finite nullstellensatz application}, in case of main branches we need to prove
      	
      	\begin{prop}\label{prop-main branch}
      		For each $k=1,\cdots,n-1, n\geq 3$, there is an integer $m_k\geq 1$ s.t. $x_k^{m_k}\in LM(J)$, with $J=\langle H_1,\dots, H_{n-1}\rangle$ ($H_i=H_{n-i}(x_{i})$) an ideal in $\Q[x_1,\cdots,x_n]$. 
      	\end{prop} 
      	The theorem suffices for proving Proposition \ref{prop-finite nullstellensatz application} since the algorithm of obtaining those leading terms involves fractions only depending on $n$ and performs identically over $\F_p$ for $p$ sufficiently larger than $n$ and all denominators of the structural coefficients used. Before proving the theorem, we study some examples for small $k\leq n$.
        \begin{exam}\label{exam-m_k}Suppose $n\geq 4$.
        	For $k=1$, $H_{1}=(n-1)x_1-(x_2+\cdots+x_n)$ gives a linear relation between all the variables and $x_1\in LM(J)$, so that we can always set $m_1=1$. 
        	
        	For $k=2$, replacing $x_1$ by $-\frac{1}{n-1}(x_2+\cdots+x_n)$, we get
        	\begin{align}\label{equation-x2}H_{2}(x_2)=G_2x_{2}^2+G_1x_2+G_0\ \mod{\langle H_{1}\rangle},
        	\end{align}
        	with $G_2,G_1,G_0$ homogeneous polynomials in $x_3,\cdots,x_n$, of degrees $0, 1, 2$ respectively. We can compute explicitly that
        	\[G_2=\dfrac{1}{n-1}(n-2)^2+{n-2\choose 2}=\dfrac{(n-2)(n^2-2n-1)}{2(n-1)}\neq 0.\]
        	Thus we can always set $m_2=2$.
        	
        	Next, we show that $m_3$ can always be set to $8$. We still replace $x_1$ by $-\frac{1}{n-1}(x_2+\cdots+x_n)$ in $H_{3}$ to get
        	\[H_{3}=K_2x_{2}^2+K_1x_2+K_0\ \mod{\langle H_{1}\rangle},\]
        	with $K_2,K_1,K_0$ homogeneous polynomials in $x_3,\cdots,x_n$, of degrees $1, 2, 3$ respectively. Using (\ref{equation-x2}) we can kill $x_2^2$ and get
        	\[H_{3}=L_1x_2+L_0\mod{\langle H_{1}, H_{2}\rangle},\]
        	with $L_1,L_0$ homogeneous polynomials in $x_3,\cdots,x_n$, of degrees $2, 3$ respectively. Now to kill $x_2$ in $H_{3}$ we need a non-linear cancellation aside with (\ref{equation-x2}) as follows:
        	\begin{align*}
        	H_{3}^2&=L_1^2x_2^2+2L_1L_0x_2+L_0^2\\
        	&=L_{1}'x_2+L'_0\mod{\langle H_{1}, H_{2}\rangle},\end{align*}
        	with $L'_1,L'_0$ homogeneous polynomials in $x_3,\cdots,x_n$, of degrees $5, 6$ respectively. Then 
        	\[L_1H_{3}^2-L'_1H_{3}=L''_0\mod{\langle H_{1}, H_{2}\rangle},\]
        	with $L''_0=L_1L'_0-L'_1L_0$ a homogeneous polynomial in $x_3,\cdots,x_n$ of degree $8$. By carrying out the detailed calculation we find $LM(L''_0)=x_3^8$.    				
        \end{exam} 
        The above example inspires us to consider general higher order non-linear cancellations likewise. Thus we introduce some extra notions besides those defined in section \ref{section-comuptational notations}, in that we need to write down the coefficients of the Hasse derivatives more explicitly. 
        
        Let $[n]$ denote the \textit{chain} $(1<2<\cdots<n)$. If $c=(j_1<\cdots<j_t)$, we say $c$ is a chain of \textit{length} $t$, denoted by $l(c)=t$. By $c\leq_{t}[n]$ we indicate that $c$ is a length $t$ sub-chain of $[n]$. We may also use $c$ to denote a \textit{multichain} $(j_1\leq j_2\leq\cdots\leq j_t)$ and $\alpha(c)=(\alpha(c)_1,\dots,\alpha(c)_n)$ to denote its occurrence vector with the part $j$ occurring $\alpha(c)_j$ times in $c$. Then  $l(c)=|\alpha(c)|$ is the total number of occurrences and if no confusion $c\leq_t[n]$ is a multichain with $l(c)=t$ and all its parts belonging to $[n]$. For short, $j\in c$ means $j$ occurs in $c$. Also, by $c_1+c_2$ we mean a derived multichain (or chain) from two multichains by collecting their parts into one. 
        
        For any vectors $\alpha,\beta\in\N^n$ we define ${\beta\choose\alpha}:={\beta_1\choose\alpha_1}\cdots{\beta_n\choose\alpha_n}$. The number of sub-multichains $c'\leq c$ with $\alpha(c')=\alpha$ prescribed is ${\alpha(c)\choose\alpha}$, which is also valid if $\alpha$ is not comparable with $\alpha(c)$ since then ${\alpha(c)\choose\alpha}$ vanishes. In addition, $\alpha\preceq\beta$ means $\alpha_j\leq\beta_j, j=1,\dots,n$. Clearly ${\beta\choose\alpha}>0$ if and only if $\alpha\preceq\beta$.

            Now direct computation on the Hasse derivative yields
        	\begin{align}\label{equation-leading coefficient}H_{n-i}^m(x_k)=&\left(\sum_{c\leq_{i}[n]}\prod_{j\in c}(x_k-x_{j})\right)^m=\sum_{c_1,\dots,c_m\leq_{i}[n]}\prod_{j\in c_1+\cdots+c_m}(x_k-x_{j})\\\notag
        	=&\sum_{c_1,\dots,c_m\leq_{i}[n]}\sum_{c\leq c_1+\cdots+c_m}(-1)^{l(c)}{\alpha(c_1)+\cdots+\alpha(c_m)\choose\alpha(c)}x^{\alpha(c)}x_k^{im-l(c)}\\\notag
        	=&\sum_{c_1,\dots,c_m\leq_{i}[n]}\sum_{\alpha\preceq\alpha(c_1+\cdots+c_m)}(-1)^{|\alpha|}{\alpha(c_1+\cdots+c_m)\choose\alpha}x^{\alpha}x_k^{im-|\alpha|}\\\notag
        	=&\sum_{\alpha\in\N^n}(-1)^{|\alpha|}\sum_{c_1,\dots,c_m\leq_{i}[n]}{\alpha(c_1+\cdots+c_m)\choose\alpha}x^{\alpha}x_k^{im-|\alpha|}\\
        	=&\sum_{0<(\alpha_1,\dots,\alpha_{k-1},0,\dots,0)=\alpha\preceq(m,\dots,m)}x^\alpha h_{\alpha,m}+\binom{n}{i}^{m}x_{k}^{im}+r_{ikm},\notag
        	\end{align}
        	where for short $0=(0,\dots,0)$ and $r_{ikm}$ summons the remaining terms with $x^\alpha<x_{k}^{im}$. For any $\alpha=(\alpha_1,\dots,\alpha_{k-1},0,\dots,0)$ with $\alpha_j\leq m$, i.e.  $x^\alpha\geq x_k^{im}$, the leading term $LT(h_{\alpha,m})$ appears as 
        	\[(-1)^{|\alpha|}\sum_{c_1,\dots,c_m\leq_{i}[n], \alpha(c_j)_k=0}{\alpha(c_1+\cdots+c_m)\choose\alpha}x_k^{im-|\alpha|}=C_{\alpha,i}(m)x_k^{im-|\alpha|},\]
        	noting that the terms with $\alpha_k\neq 0$ sum up to zero. It is worth to mention the easy observation that for any two symmetric vectors $\alpha=\sigma\cdot\alpha'$, i.e. $\alpha$ can be obtained by permuting the coordinates of $\alpha'$ using some $\sigma\in S_{k-1}$, we always have $C_{\alpha,i}(m)=C_{\alpha',i}(m)$. Conversely one easily checks that identical columns corresponds to symmetric $\alpha$'s. 
           
        	Clearly those leading coefficients $C_{\alpha,i}(m)$ do not vanish. If $\alpha\preceq\alpha'$, then the multichains $c$ with $\alpha'\preceq \alpha(c)$ also satisfy $\alpha\preceq\alpha(c)$ whence $|C_{\alpha,i}(m)|\geq|C_{\alpha',i}(m)|$. Thus $C_{\alpha,i}(x_k)$ attains maximum only when $\alpha=0$ which is 
        	\[C_{0,i}(m)=\sum_{c_1,\dots,c_m\leq_{i}[n],\alpha(c_j)_k=0}{\alpha(c_1+\cdots+c_m)\choose0}={n-1\choose i}^m,\]
        	i.e. the coefficient of $x_k^{im}$. The minimum is attained when $\alpha_1=\cdots=\alpha_{k-1}=m$ which is ${n-k+1\choose i-k+1}^m$ (vanishes if $i<k-1$). 
         
         The above expression of $C_{\alpha,i}(m)$ is equivalent to using the multivariate Fa\`{a} di Bruno's formula noting that $C_{\alpha,i}(m)$ is nothing but a multiple of $\frac{\partial x^{im}}{\partial x_1^{\alpha_1}\cdots\partial x_{k-1}^{\alpha_{k-1}}\partial x_k^{im-|\alpha|}}H_{n-i}^m(x_k)$.
         However, a computationally more accessible formula is given by the following Combinatorial Nullstellensatz as of \cite{Karasev-Petrov}.
         \begin{prop}\label{prop-combinatorial nullstellensatz}
         	For any $f\in K[x_1,\dots,x_n]$ of degree $|\alpha|$ over an arbitrary field $K$, the coefficient of $x^\alpha$ in $f$ has the following expression:
         	\[[x^\alpha]f(x_1,\dots,x_n)=\sum_{b_j\in A_j}\dfrac{f(b_1,\dots,b_n)}{\varphi_1'(b_1)\cdots\varphi_n'(b_n)},\]
         	where $A_j\subset K$ are any subsets of size $\alpha_j+1$ and $\varphi_j(x)=\prod_{b\in A_j}(x-b)$.
         \end{prop}
         If we choose $A_j=\{0,1,\dots,\alpha_j\}$ for $j\leq k-1$ and $A_{k+1}=\cdots=A_n=\{0\}$, the above Nullstellensatz promises         
         \begin{equation}\label{equation-combinatorial nullstellensatz}C_{\alpha,i}(m)=\sum_{b_j\leq\alpha_j}\dfrac{\left(H_{n-i}^m(x_k)\right)(b_1,\dots,b_{k-1},1,0,\dots,0)}{\prod_{j=1}^{k-1}\prod_{b_j\neq b\leq \alpha_j}(b_j-b)}.\end{equation}
         More significantly, it implies the following arithmetic on $C_{\alpha,i}(m)$ which is crucial to our later proof.
          
          \begin{lem}\label{lem-arithmetic on coefficients}
          	Keep notations above and gather $\alpha=(\alpha_{1},\dots,\alpha_{k-1},0,\dots,0)\in\N^n$ with $\alpha_j\leq m_j\in\Z_+$, $j=1,\dots, k-1$, no two of which can be identified by permuting their first $k-1$ coordinates. Denote by $N$ the number of such vectors. Then for any $M\in\N$, the $N$ by $N$ square matrix $(C_{\alpha,i}(m))$ with $M+1\leq m\leq M+N$  is non-degenerate.          	
          \end{lem}
      \begin{proof}
      	For $j\leq k-1$, choose $A_{j}=\{b_{j,0},\dots,b_{j,m_j}\}\subset\Q$ of $m_j+1$ numbers, 
        such that the values $H_{n-i}(b_1,\dots,b_{k-1},1,0,\dots,0)\neq0$ are all distinct for different $(b_1,\dots,b_{k-1})\in A_1\times\cdots\times A_{k-1}$. (This is possible since the condition defines an open subset of $\Q^{k-1}$.)
      	
      	Now following the Combinatorial Nullstellensatz, we can write
      	$$(C_{\alpha,i}(M+l))=\left(H_{n-i}^{M+l}(b_1,\dots,b_{k-1}, 1,0,\dots,0)\right)\left(\Phi_\alpha\right),$$
      	where $\left(H_{n-i}^{M+l}(b_1,\dots,b_{k-1},1,0,\dots,0)\right)_{1\leq l\leq N, b_j\in A_j}=:H$ is an $N$ by $(m_1+1)$ $\cdots(m_{k-1}+1)$ matrix, and $\left(\Phi_\alpha\right)$ is a matrix of $N$ columns. Here corresponding to the formula of (\ref{equation-combinatorial nullstellensatz}), for each $\alpha$, $\Phi_\alpha=(\phi^\alpha_{b_1,\dots,b_{k-1}})$ is designated to produce $C_{\alpha,i}(M+l)$ by multiplying the $l$-th row of $H$ for any $l\leq N$. Thus $\phi^\alpha_{b_1,\dots,b_{k-1}}=\dfrac{1}{\varphi_{\alpha,1}'(b_1)\cdots\varphi_{\alpha,k-1}'(b_{k-1})}$ for $b_j$ ranging from $b_{j,0}$ to $b_{j,\alpha_j}$, otherwise $\phi^\alpha_{b_1,\dots,b_{k-1}}=0$, in which $\varphi_{\alpha,j}(x)=\prod_{r=0}^{\alpha_j}(x-b_{j,r})$. 
      	
      	By our choice of $A_j$ and noting that the number of columns of $H$ is generally larger than $N$, any $N$ by $N$ minor sub-matrix of $H$ is a Vandermonde matrix, hence $H$ has rank $N$. If we can show $(\Phi_\alpha)$ also has rank $N$, then $(C_{\alpha,i}(M+l))$ is non-degenerate (of rank $N$). Suppose there exists linear dependence: $\sum f_\alpha\Phi_\alpha=0$. Pick all the columns with $f_\alpha\neq 0$ and find all the maximal ones among them along $\prec$ which are all unique. Say $\beta$ is maximal, then its (lowest) entry in the row indexed by $\beta_1,\dots,\beta_{k-1}$ is the only nonzero entry in this row among all the picked columns, hence $f_\beta$ must be zero, a contradiction.
      \end{proof}
      \begin{remark}
      	Employing generalized Vandermonde matrices, the matrix $(C_{\alpha,i}(m))$ may be shown non-degenerate for cases where $N$ positive integers $m$ are not necessarily consecutive. Also note that the hypothesis on symmetry is restricted to $\alpha$ with $\alpha_j\leq m_j$. For example, if $m_1=2, m_2=4$, then $x_1^2x_2^3$ is not symmetric to $x_1^3x_2^2$ since the latter is not in our consideration.
      \end{remark}
      	\begin{proof}[Proof of Proposition \ref{prop-main branch}]
      		Let $J_k=\langle H_{1},\cdots,H_{k}\rangle$, $1\leq k\leq n-1$. We want to show by induction, for all $k\leq n-1$ $J_k$ contains homogeneous polynomials $g_1,\cdots,g_k$ with leading terms $LT(g_i)=x_i^{m_i}$ for some $m_i\geq 1, i=1,\cdots,k$, and $g_i$'s are symmetric in $x_{i+1},\cdots,x_{n}$, i.e. $g_i=x_i^{m_i}+a_{m_i-1}x_i^{m_i-1}+\cdots+a_0$ with $a_l\neq 0$ being symmetric in $x_{i+1},\cdots,x_n$ for $l\leq m_{i}-1$. Note that we can always set $m_1=1,m_2=2,m_3=8$ by Example \ref{exam-m_k}.
      		      		
      		Assuming the cases for $1\leq k-1$ ($\leq n-2$) we need to verify it for $k$. Let $\bar{m}_{k-1}=(m_1,\cdots,m_{k-1},0,\cdots,0)\in\N^n$. By (\ref{equation-leading coefficient}) we write for any $m\in\Z_+$
      		\[H_{k}^m=\sum_{\alpha=(\alpha_1,\dots,\alpha_{k-1},0,\dots,0)>(0,\dots,0)}x_1^{\alpha_1}\cdots x_{k-1}^{\alpha_{k-1}}h_{\alpha,m}+\binom{n}{k}^{m}x_{k}^{km}+r_{m},\]
      		where $r_m$ collects the monomials smaller than $x_k^{km}$ in lexicographic order. For any term with $\alpha\geq \bar{m}_{k-1}$, say $\alpha_i\geq m_i$ for some $i\leq k-1$, we can replace $x_i^{\alpha_i}$ by $x_i^{\alpha_i-m_i}(g_i-x_i^{m_i})$ when modulo $J_{k-1}$. Since $g_i$ is symmetric in $x_{k},\dots,x_{n}$, $g_i-x_i^{m_i}$ can not have $x_{k}^{m_i}$ as leading monomial. Thus the replacement does not affect $x_k^{km-|\alpha|}$ as the leading monomial of $h_{\alpha,m}$ by (\ref{equation-leading coefficient}). After all such replacements until there is no $\alpha_i\geq m_i$ for any $i\leq k-1$,
      		\begin{equation}\label{equation-reduced H}H_{k}^m=\sum_{(0,\dots,0)<\alpha\prec \bar{m}_{k-1}}x^{\alpha}h'_{\alpha,m}+\binom{n}{k}^{m}x_{k}^{km}+r_{m}\mod J_{k-1},\end{equation}
      		where each $h'_{\alpha,m}$ is symmetric in $x_{k},\dots,x_n$ and  $LM(h'_{\alpha,m})=LM(h_{\alpha,m})=x_k^{km-|\alpha|}$. Recall that $\prec$ denotes for each coordinate of the left vector being zero or strictly less than that of the right respectively.	
      		
      		Now similar to Gaussian elimination, by row reduction we may kill the terms with those $\alpha$, i.e. by performing \begin{align*}&h'_{\alpha',m_{1}}H_k^{m_{2}}-h'_{\alpha',m_{2}}H_k^{m_{1}}\\
      		=&\sum x^{\alpha}\left(h'_{\alpha',m_{1}}h'_{\alpha,m_{2}}-h'_{\alpha',m_{2}}h'_{\alpha,m_{1}}\right)\\
      		&+h'_{\alpha',m_{1}}\left(\binom{n}{k}^{m_{1i}}x_{k}^{km_{2}}+r_{m_{2}}\right)-h'_{\alpha',m_{2}}\left(\binom{n}{k}^{m_{11}}x_{k}^{km_{1}}+r_{m_{1}}\right)\end{align*}
      		to kill the term with $\alpha'$ in $H_k^{m_2}$ (modulo the ideal $J_{k-1}$). Note that $h'_{\alpha',m_{1}}h'_{\alpha,m_{2}}$ and $h'_{\alpha',m_{2}}h'_{\alpha,m_{1}}$ have identical leading monomial $x_k^{k(m_{1}+m_{2})-|\alpha'|-|\alpha|}$for $\alpha\neq\alpha'$. Suppose their leading coefficients do not coincide, we can proceed likewise to kill terms with $\alpha_2$ and so on until $\alpha_t$ is killed if possible, and we are done with the proof. 
      		
      		In the process, the Gaussian elimination of leading terms is equivalent to that of leading coefficients $C_{\alpha}(m)$, which leads us to study the matrix $C=(C_{\alpha}(m))_{m\in\N}$. Note that if $\alpha'$ and $\alpha$ are symmetric, their corresponding columns are identical so that the matrix becomes degenerate. However, any row reductions performed on the two columns are also identical. Thus if one is killed so is the other. This suggests what we should really study is the matrix $\tilde{C}:=S_n\backslash C$, i.e. the symmetric (identical) columns of $C$ are assimilated. Then $\tilde{C}$ fits to the hypothesis of Lemma \ref{lem-arithmetic on coefficients}, and its full minors of consecutive rows have full rank so that the Gauss elimination is promised to kill all terms with $\alpha\prec\bar{m}_{k-1}$ for rows with large enough indexes $m$. Choose the smallest such $m$ as $m_k$ and resulted $H_k^{m_k}$ as $g_k$ (uniformed to be monic if necessary). The symmetry of $g_k$ in $x_{k+1},\dots,x_n$ is due to that of $H_k$. Hence we are done with the induction step and the theorem follows.
     	  	\end{proof}
      	  	\begin{remark}
      	  	For $p$ larger than the denominators of any multipliers used in the Gauss elimination, the proof works over $\F_p$ as well to establish Proposition \ref{prop-finite nullstellensatz application}. Our algorithm is a specialization of Noether normalization as in Proposition \ref{prop-leading power in Noether normalization}.\end{remark}
        	
           By symmetry of roots, for any main branch with distinct indexes $i_1,\cdots,i_{n-1}$, the same proof above works for the alphabetical order $x_{i_1}>x_{i_2}>\cdots>x_{i_{n-1}}>x_j$ in which $\{j\}=\{1,\cdots,n\}\smallsetminus\{i_1,\cdots,i_{n-1}\}$.
           \begin{cor}\label{cor-main branches}
    	    For any $J=\langle H_{n-1}(x_{i_1}),\cdots,H_{1}(x_{i_{n-1}})\rangle$ with $i_1,\cdots,i_{n-1}$ all distinct, there exist $m_k\in\Z_+$ such that $x_{i_k}^{m_k}\in LM(J)$ over $\Q$, for $k=1,\cdots,n-1$.
            \end{cor} 
             This proves Theorem \ref{thm-general branches} for all main branches. Under the rearranged alphabetical order, the proof of Proposition \ref{prop-main branch} works regardless of choice of derivatives, i.e.
             \begin{cor}\label{cor-derivative choice neglected}
             	For any $1\leq j_1<\cdots<j_k\leq n-1$ and $1\leq i_1,\dots, i_k\leq n$ distinct, there are $m_l\in\Z_+, l=1,\dots, k$ such that $x_{i_l}^{m_l}\in LM(\langle H_{n-j_1}(x_{i_1}),\dots,H_{n-j_k}(x_{i_k})\rangle)$ over $\Q$. 
             \end{cor}
             
             Complying with Proposition \ref{prop-variables in Noether normalization}, we have $\langle H_{n-j_1}(x_{i_1}),\dots,H_{n-j_k}(x_{i_k})\rangle\cap \Q[u]=0$ for $u=\{x_1,\dots,x_n\}\smallsetminus\{x_{i_1},\dots,x_{i_{k}}\}$. Hence by Cohen-Macaulayness, 
             \[ht(\langle H_{n-j_1}(x_{i_1}),\dots,H_{n-j_k}(x_{i_k})\rangle)=n-\dim(\Q[x_1,\dots,x_n]/I)=n-(n-k)=k,\]
             where $ht()$ denotes the height of an ideal. In other words, $H_{n-j_1}(x_{i_1}),\dots,H_{n-j_k}(x_{i_k})$ form a regular sequence which also follows from Corollary \ref{cor-main branches}.
             
             For general branches defined by $J=\langle H_{n-1}(x_{i_1}),\cdots,H_{1}(x_{i_{n-1}})\rangle$ with $i_1,\dots,i_{n-1}$ not necessarily distinct, we may still obtain results as of Corollary \ref{cor-main branches} through linearized Noether normalization as $(4)$ of Proposition \ref{prop-leading power in Noether normalization}.
             If the number of distinct indexes occurring is $k\leq n-1$, by symmetry we may assume $\{i_1,\dots,i_{n-1}\}=\{1,\dots,k\}$ so that our algorithm in the proof of Proposition \ref{prop-main branch} proceeds as well. By Corollary \ref{cor-derivative choice neglected} we obtain $x_1^{m_1},\dots,x_k^{m_k}\in LM(J)$, say by working on the sub-ideal $J_0=\langle H_{n-{j_1}}(x_1),\dots,H_{n-j_k}(x_k)\rangle$. For $u_0=\{k+1,\dots,n\}$ we have $J_0\cap\Q[u_0]=0$ and $ht(J_0)=k$. 
             
             Let $l_0=\{j_1,\dots,j_k\}$. For any $j\notin l_0$, if we can show that the intersection between $J_1=J_0+\langle H_{n-j}(x_{i_j})\rangle$ and $\Q[u_0]$ is not zero, then by Proposition \ref{prop-variables in Noether normalization}, $ht(J_1)=k+1$ and $J_1\cap\Q[u_1]=0$ for some $u_1\subset u_0$ of size $n-k-1$. Applying Proposition \ref{prop-leading power in Noether normalization} (modulo an isomorphism) we are guaranteed to have for some $i_{k+1}\in u_0$ such that $x_{i_{k+1}}^{m_{k+1}}\in LM(J_1)$. 
             Subsequently update $l_1:=l_0\cup\{j\}$ and $J_2:=J_1+\langle H_{n-j'}(x_{i_{j'}})\rangle$ for any $j'\notin l_1$. By further investigating $J_1\cap\Q[u_1]$, we may determine whether the height of $J_2$ grows. If each iteration of the process raises the height by one, at the end we may conclude that $ht(J)=n-1$. We consolidate this hypothetical procedure using Lemma \ref{lem-arithmetic on coefficients} as follows.
             \begin{prop}\label{prop-general branches}
             	Any ideal $J=\langle H_{n-1}(x_{i_1}),\cdots,H_{1}(x_{i_{n-1}})\rangle$ has $ht(J)=n-1$. Consequently $\dim Z(J)=1$ and Theorem \ref{thm-general branches} follows. 
             \end{prop}
         \begin{proof}
         	With the notations above, we verify that $ht(J_1)\cap\Q[u_0]\neq 0$. Applying Lemma \ref{lem-arithmetic on coefficients} to $\alpha=(\alpha_1,\dots,\alpha_k,0,\dots,0)\in\N^n$ with confinement, we see that $(C_{\alpha,j}(m))$ is non-degenerate, for $M<m\leq M+N$ with $M\in\Z_+$ large enough and $N$ being the number of $\alpha$ which are not symmetric to each other. Write $H^m_{n-j}(x_{i_j})$ in the form below \begin{equation}\label{equation-reduced H in general}H^m_{n-j}(x_{i_j})=\sum_{(0,\dots,0)<\alpha\prec (m_1,\dots,m_k,0,\dots,0)}x^{\alpha}h_{\alpha,m}+r_{m}\mod J_{0},\end{equation} 
           where $h_{\alpha,m}$ and $r_m$ are polynomials symmetric in $u_0$. Then we similarly kill the terms only involving $x_{i_1},\dots,x_{i_k}$ by Gauss elimination. After the elimination, the residue terms involving $r_m$ do not vanish similarly because of non-degeneracy of $(C_{\alpha,j}(m))$ for $\alpha=(\alpha_1,\dots,\alpha_{k+1},0,\dots,0)$ by Lemma \ref{lem-arithmetic on coefficients}. Thus $J_1\cap\Q[u_0]\neq 0$ indeed and $ht(J_1)=k+1$.
           
           Now by (4) and (2) of Proposition \ref{prop-leading power in Noether normalization}, we can choose a linear transform $\varphi_1: \Q[x_1,\dots,x_n]\rightarrow\Q[x_1,\dots,x_n]$ such that $\tilde{x}_1^{m_1},\dots,\tilde{x}_{k+1}^{m_{k+1}}\in\varphi(J_{1})$ for some $m_j\in\Z_+$ and $\tilde{x}_j$ the image of $x_j$. Then guaranteed by Proposition \ref{prop-leading power in Noether normalization} and Lemma \ref{lem-arithmetic on coefficients} the iteration proceeds until at step $n-k$ when we choose a linear transform $\varphi_{n-k}:\Q[x_1,\dots,x_n]\rightarrow\Q[x_1,\dots,x_n]$ such that $\varphi_{n-k}(J)\cap\Q[u_{n-k-1}]\neq 0$ for some $u_{n-k-1}\subset\{x_1,\dots,x_{n}\}$ of size two. Then we see $ht(J)=n-1$ and $\dim(Z(J))=1$, which proves Theorem \ref{thm-general branches}.
         \end{proof}
     
         \begin{proof}[Proof of Theorem \ref{thm-arbitrary derivatives}]
         	Guaranteed by the non-degeneracy of coefficients matrices as in Lemma \ref{lem-arithmetic on coefficients}, the above proof works for ideals generated by Hasse derivatives of not necessarily distinct degrees. The similar process of Gauss elimination and linear transforms promises that \[ht(f^{(j_1)}(x_{i_1}),\dots,f^{(j_{n-1})}(x_{i_{n-1}}))=n-1,\]  
         	for any $n-1$ arbitrary distinct pairs $(i_k,j_k)$ with $1\leq i_{k}\leq n$ and $1\leq j_k\leq n-1$.
         \end{proof}

 \end{document}